\documentclass{amsart}

\usepackage{amssymb}

\usepackage[utf8]{inputenc}

\newcommand{\Aut}{\ensuremath{\operatorname{Aut}}}

\newcommand{\Iso}{\ensuremath{\operatorname{Iso}}}
\newcommand{\GL}{\ensuremath{\operatorname{GL}}}
\newcommand{\SL}{\ensuremath{\operatorname{SL}}}
\newcommand{\G}{\ensuremath{\Gamma}}
\newcommand{\E}{\ensuremath{\operatorname{E}}}
\newcommand{\Or}{\ensuremath{\operatorname{O}}}

\newcommand{\R}{\ensuremath{\mathbb R}}

\newcommand{\Z}{\ensuremath{\mathbb Z}}

\newcommand{\N}{\ensuremath{\mathbb N}}

\newcommand{\ses}[5]{#1 \longrightarrow #2 \longrightarrow #3 \longrightarrow #4 \longrightarrow #5}

\newcommand{\st}{\ensuremath{|\;}}

\newtheorem{thm}{Theorem}
\newtheorem{prop}[thm]{Proposition}

\newtheorem{cor}{Corollary}
\theoremstyle{remark}
\newtheorem{rem}{Remark}
\theoremstyle{definition}
\newtheorem{df}{Definition}
\newtheorem{ex}{Example}

\DeclareMathOperator{\diag}{diag}

\author{Rafał Lutowski}
\thanks{The author was supported by the National Science Center Poland grant No. 2013/09/B/ST1/04125.}
\title{Irreducible euclidean representations of Fibonacci groups}

\begin{document}

\maketitle

\begin{abstract}
We show that every Hantzsche-Wendt group is an epimorphic image of a certain Fibonacci group. 
\end{abstract}

\section{Introduction}

Let $n \in \N$. A group $\G$ is called an \emph{$n$-dimensional crystallographic} if it is a discrete and cocompact subgroup of $\E(n) = \Or(n) \ltimes \R^n$ -- the group of isometries of the $n$-dimensional euclidean space. By the Bieberbach theorems (see \cite{Bi11},\cite{Bi12}, \cite[Theorem 2.1]{Sz12}) $\G$ fits into the following short exact sequence
\begin{equation}
\label{eq:cryst}
\ses{0}{\Z^n}{\G}{G}{1},
\end{equation}
where $G \subset \GL(n,\Z)$ is a finite group, so called holonomy group of $\G$ and $\Z^n$ is a faithful $G$-module, with the action defined by the left matrix multiplication.

A torsionfree crystallographic group $\G \subset \E(n)$ is called a \emph{Bieberbach group}. In this case the orbit space $X = \R^n/\G$ is a flat manifold -- a closed connected Riemannian manifold with constant sectional curvature equal to zero. Moreover $\G \cong \pi_1(X)$.

A Bieberbach group $\G$, defined by the short exact sequence \eqref{eq:cryst}, is a \emph{Hantzsche-Wendt group} if $G \subset \SL(n,\Z)$ and $G \cong C_2^{n-1}$, where $C_2$ is a cyclic group of order $2$, i.e. $G$ is an elementary abelian $2$-group of rank $n-1$. An underlying manifold $\R^n/\G$ is called a \emph{Hantzsche-Wendt manifold}. Hantzsche-Wendt groups and manifolds exist only in odd dimensions greater or equal than $3$ (see \cite[page 2]{MR99}). In dimension $3$ there is only one Hantzsche-Wendt manifold, which is the only $3$-dimensional orientable flat manifold. The non-extensive list of results concerning Hantzsche-Wendt groups and manifolds contains
\begin{itemize}
\item homological description of Hantzsche-Wendt manifolds -- all of them are rational homology spheres (see \cite{Sz83});
\item the form of the holonomy group -- we can always find a HW-group, isomorphic to the given one, such that its holonomy group contains diagonal matrices only (see \cite{RS05});
\item abelianization -- starting from dimension 5 the abelianization of a HW-group is isomorphic to its holonomy group (see \cite{P07}).
\end{itemize}

Let $r,n \in \N$. A Fibonacci group $F(r,n)$ is a group with the presentation
\[
\begin{split}
F(r,n) = \langle a_0, a_1, \ldots, a_{n-1} \st 
& a_0 a_1 \ldots a_{r-1} = a_r,\\
& a_1 a_2 \ldots a_{r} = a_{r+1},\\
& \vdots \\
& a_{n-1}a_0 \ldots a_{r-2} = a_{r-1} \rangle.\\
\end{split}
\]
The following facts illustrate a connection between Fibonacci groups and geometry:
\begin{itemize}
\item $F(2,6)$ is isomorphic to the 3-dimensional Hantzsche-Wendt group;
\item for $n \geq 4$ there exists a closed hyperbolic manifold with fundamental group isomorphic to $F(2,2n)$ (see \cite{HKM98});
\item for odd $n$ the Fibonacci group $F(2,n)$ cannot be a fundamental group of any hyperbolic 3-orbifold of finite volume (see \cite{M95}).
\end{itemize}

In the paper \cite{Sz01} a connection between Fibonacci and Hantzsche-Wendt groups is presented. Namely, for every odd $n \geq 3$ there exists HW-group of dimension $n$ which is an epimorphic image of the Fibonacci group $F(n-1,2n)$. In the review of the paper Juan Pablo Rossetti points out that there may exist epimorphisms onto other HW-groups as well. In our paper we prove that in fact every $n$-dimensional HW-group is an image of the group $F(n-1,2n)$, i.e. all HW-groups of the same dimension share the same non-trivial relations.

In Section 2 we present a way to decompose euclidean representations, which in general are not linear. We use this procedure to decompose (the identity map on) every Hantzsche-Wendt group. In the following section we show a construction of an epimorphism of the group $F(n-1, 2n)$ to a certain one dimensional crystallographic group, for every odd $n \geq 3$. We use this construction to prove our main theorem in Section 4. 

\section{Decompositions of euclidean representations}

\begin{df}
An \emph{euclidean representation} of a group $G$ is a homomorphism $\varphi \colon G \to \E(n)$.
\end{df}

\begin{ex}[{\cite[Theorem 1]{Sz01}}]
\label{ex:cyclichw}
Let $n \geq 3$ be an odd integer. Let $\Gamma_n$ be a HW-group generated by the elements $(B_i,b_i) \in \E(n)$, where
\[
B_i = \diag(\underbrace{-1,\ldots,-1}_{i-1}, 1, -1, \ldots, -1)
\]
and
\[
b_i = \frac{1}{2}(e_i+e_{i+1})
\]
for $i=1,\ldots,n-1$ and $e_1,\ldots,e_n$ being the standard basis' vectors. Then there exists an epimorphism
\[
\Phi_n \colon F(n-1, 2n) \to \Gamma_n \subset \E(n).
\]
\end{ex}

\begin{rem}
If $V = V_1 \oplus V_2$ is a direct sum decomposition of a vector space $V$ and $f_i \colon V_i \to V_i$ are \emph{maps} of vector spaces, for $i=1,2$, then $f_1 \oplus f_2 \colon V \to V$ is a map defined by the following formula
\[
\forall_{v_1 \in V_2} \forall_{v_2 \in V_2} \; (f_1 \oplus f_2) (v_1 + v_2) = f_1(v_1) + f_2(v_2).
\] 
In addition if we have representations $\varphi^{(i)} \colon G \to \Iso(V_i)$ of a group $G$ to the groups of isometries of spaces $V_i$, for $i=1,2$, we can say about direct sum of those representations:
\[
\forall_{g \in G} \forall_{v_1 \in V_2} \forall_{v_2 \in V_2} (\varphi^{(1)} \oplus \varphi^{(2)})_g(v_1 + v_2) = \varphi^{(1)}_g(v_1) + \varphi^{(2)}_g(v_2).
\]
\end{rem}

The above decomposition will be useful for our purposes, but we have to remember that euclidean representations are not linear in general. The following example shows that in our case even irreducible representations can be decomposed in the spirit given above.

\begin{ex}
Let $\varphi_1,\varphi_2 \colon \Z \oplus \Z \to \E(1) = \Or(1) \ltimes \R$ be one dimensional euclidean representations, given by
\begin{align*}
\varphi_1(1,0) = (1,1), \quad \varphi_1(0,1) = (1,0), \\
\varphi_2(1,0) = (1,0), \quad \varphi_2(0,1) = (1,1).
\end{align*}
Then the direct sum of the two representations 
\[
\varphi_1 \oplus \varphi_2 \colon \Z \oplus \Z \to \E(2) = \Or(2) \ltimes \R^2
\]
is defined as follows
\[
(\varphi_1 \oplus \varphi_2)(1,0) = 
\left(
\begin{bmatrix}
1 & 0\\
0 & 1
\end{bmatrix},
\begin{bmatrix}
1 \\
0 
\end{bmatrix}
\right)
\text{ and }
(\varphi_1 \oplus \varphi_2)(0,1) = 
\left(
\begin{bmatrix}
1 & 0\\
0 & 1
\end{bmatrix},
\begin{bmatrix}
0 \\
1 
\end{bmatrix}
\right).
\]
It is obvious that the representation $\varphi_1 \oplus \varphi_2$ does not have any proper invariant subspace -- it is irreducible.
\end{ex}

The following proposition is in fact a recipe for decomposing euclidean representations. Let $n \in \N$ and let $r \colon \E(n) \to \Or(n)$ be the rotational homomorphism:
\[
\forall_{(B,b) \in \E(n)} \; r(B,b) = B.
\]
Let $\varphi \colon \Gamma \to E(n)$ be an euclidean representation of a group $\Gamma$. We get a linear representation $r \varphi \colon \Gamma \to \Or(n)$. Let
\begin{equation}
\label{eq:decomp}
\R^n = V_1 \oplus \ldots \oplus V_k
\end{equation}
be a decomposition of $r \varphi$ with corresponding projections $p_i \colon \R^n \to V_i$ for $i=1,\ldots,k$.

\begin{prop}
We have
\[
\varphi = \varphi^{(1)} \oplus \ldots \oplus \varphi^{(k)}.
\]
In the above formula $\varphi^{(i)} \colon \Gamma \to \Iso(V_i)$ is given by
\[
\forall_{v \in V_i} \varphi_g^{(i)} (v) = (B, p_i(b))v = Bv + p_i(b),
\]
where $g \in \Gamma, \varphi_g = (B,b) \in \E(n)$ and $1 \leq i \leq k$.
\end{prop}

\begin{rem}
Note that $r \varphi$ is a unitary linear representation and hence the decomposition \eqref{eq:decomp} can be made in such a way that the constituents are irreducible.
\end{rem}


Our goal is to decompose every HW-group in the above fashion. We begin with a description of a form of any HW-group.

\begin{thm}[{\cite[Theorem 3.1]{RS05}}]
\label{thm:hwform}
Let $\Gamma$ be an $n$-dimensional HW-group. Then
\[
\Gamma \cong \langle (B_i, b_i) \in \E(n) \st 1 \leq i < n \rangle \subset \E(n),
\]
where $B_i$ are the matrices given in Example \ref{ex:cyclichw} and $b_i \in \frac{1}{2}\Z^n$ for $1 \leq i < n$.
\end{thm}

\begin{rem}
The group $\langle (B_i, b_i) \in \E(n) \st 1 \leq i < n \rangle$ is crystallographic and hence it is a Hantzsche-Wendt group. From now on we will assume that every HW-group is in the above form.
\end{rem}

\begin{cor}
\label{cor:hwdecomp}
Let $n \in \N$ be odd. Let $\Gamma$ be an $n$-dimensional HW-group. Then the identity map on $\G$ admits the decomposition
\[
id_\Gamma = \varphi^{(1)} \oplus \ldots \oplus \varphi^{(n)}.
\]
For every $1 \leq i \leq n$ the homomorphism $\varphi^{(i)} \colon \Gamma \to E(1)$ is given by
\[
\varphi^{(i)}(C,c) = (C_i, c_i),
\]
where 
\[
(C,c) =
\left(
\begin{bmatrix}
C_1 & & \\
 & \ddots &\\
 & & C_n
\end{bmatrix},
\begin{bmatrix}
c_1\\
\vdots\\
c_n
\end{bmatrix}
\right)
\in \Gamma.
\]
\end{cor}

\section{One dimensional representations of the Fibonacci groups}

This section is devoted to the construction of epimorphisms of the Fibonacci group $F(n-1,2n)$ onto a subgroup of $\E(1)$ of a certain type.

\begin{thm}
\label{thm:onedimrep}
Let $n \geq 3$ be an odd integer.
Let $d_0,\ldots,d_{n-2} \in \R$ and let $\Gamma$ be a one dimensional crystallographic group generated by the elements
\begin{equation}
\label{eq:ds}
D_0 = (1,d_0), D_1 = (-1,d_1), \ldots, D_{n-2} = (-1,d_{n-2}).
\end{equation}
Then there exists an epimorphism
\[
\Phi \colon F(n-1,2n) \to \Gamma,
\]
such that $\Phi(a_i) = D_i$ for $i=0,\ldots,n-2$.
\end{thm}
\begin{proof}
If the map $\Phi$ is a group homomorphism then obviously it is an epimorphism. It is enough for us to show that the recursively defined sequence $(D_i)$
\begin{equation}
\label{eq:dseq}
\forall_{i \geq 0} D_{i+n-1} = D_i \cdot \ldots \cdot D_{i+n-2},
\end{equation}
where the elements $D_0,\ldots,D_{n-2}$ are defined in \eqref{eq:ds}, is periodic with period $2n$. Equivalently it is enough to prove that
\[
D_0 = D_{2n}, \ldots, D_{n-2} = D_{3n-2}.
\]
Note that for every $i > 0$ we have the following formula
\begin{equation}
\label{eq:addrel}
D_{i+n-1} = D_i \cdot \ldots \cdot D_{i+n-2} = D_{i-1}^{-1} D_{i-1} D_i \ldots D_{i+n-3} D_{i+n-2} = D_{i-1}^{-1}D_{i+n-2}^2.
\end{equation}
We get:
\begin{align*}
D_{n-1} &= D_0 \ldots D_{n-2} = (-1,d_0 + d_1 - d_2 + \ldots + d_{n-2}) = (-1, d_{n-1})\\
D_n     &= D_0^{-1} D_{n-1}^2 = (1, d_0)^{-1} (-1, d_{n-1})^2 = (1, -d_0)\\
D_{n+1} &= D_1^{-1} D_n^2 = (-1, d_1)^{-1} (1, -d_0)^2 = (-1, d_1 - 2d_0)\\
D_{n+2} &= D_2^{-1} D_{n+1}^2 = (-1, d_2)^{-1} (-1, d_1 -2d_0)^2 = (-1, d_2) = D_2 \\
D_{n+3} &= D_3^{-1} D_{n+2}^2 = (-1, d_3)^{-1} (-1, d_2)^2 = (-1, d_3) = D_3\\
\vdots \\
D_{n+k} &= D_k^{-1}D_{n+k-1}^2 = (-1, d_k)^{-1} (-1, d_{k-1})^2 = (-1, d_k) = D_k\\
\vdots \\
D_{2n}  &= D_n^{-1} D_{2n-1}^2 = (1, -d_0)^{-1} (-1, d_{n-1})^2 = (1, d_0) = D_0\\
D_{2n+1}&= D_{n+1}^{-1} D_{2n}^2 = (-1, d_1 - 2d_0)^{-1} (1, d_0)^2 = (-1, d_1) = D_1\\
D_{2n+2}&= D_{n+2}^{-1}D_{2n+1}^2 = D_2^{-1} D_1^2 = D_2\\
\vdots \\
D_{2n+k}&= D_{n+k}^{-1}D_{2n+k-1}^2 = D_k^{-1}D_{k-1}^2 = D_k
\end{align*}
In the above calculations $3 \leq k \leq n-1$, hence the claim follows.
\end{proof}

\begin{cor}
\label{cor:onedimepi}
Let $n \geq 3$ be an odd integer.
Let $\Gamma \subset \E(1)$ be a one dimensional crystallographic group generated by the elements
\begin{equation}
\label{eq:onedimgens}
\setlength{\arraycolsep}{.1em}
\begin{array}{rcl}
D_0 = (-1, d_0), \ldots, D_{k-1} &=&(-1, d_{k-1}),\\
D_k &=& (1, d_k),\\
D_{k+1} &=& (-1, d_{k+1}), \ldots, D_{n-2} = (-1, d_{n-2}),
\end{array}
\end{equation}
where $0 \leq k \leq n-1$. Then there exists an epimorphism $\Phi \colon F(n-1, 2n) \to \Gamma$ such that $\Phi(a_i) = D_i$ for $i=0, \ldots, n-2$.
\end{cor}
\begin{proof}
Let $(D_i)$ be a sequence of elements of $\G$ defined recursively as in \eqref{eq:dseq}.
By Theorem \ref{thm:onedimrep} there exists an epimorphism $\Phi' \colon F(n-1, 2n) \to \Gamma$ such that
\[
\forall_{0 \leq i < 2n}
\Phi'(a_i) = D_{k+i}.
\]
Now let $\sigma \in \Aut(F(n-1,2n))$ be the automorphism sending $a_i$ to $a_{i-1}$ (subscripts are taken modulo $2n$). 
Let $\Phi = \Phi' \sigma^k$. We get
\[
\forall_{0 \leq i < 2n}
\Phi(a_{k+i}) = \Phi'\sigma(a_{k+i}) = \Phi'(a_i) = D_{k+i}.
\]
Similarly as in \eqref{eq:addrel} we have the following relations in $F(n-1,2n)$:
\[
\forall_{0 \leq i < 2n} a_i a_{i+n} = a_{i+n-1}^2.
\]
Using equation \eqref{eq:addrel} again we have
\[
\forall_{0 \leq i < k} \Phi(a_i) = \Phi(a_{i+n-1}^2a_{i+n}^{-1}) = D_{i+n-1}^2D_{i+n}^{-1} = D_i.
\]
Summing up we get a homomorphism $\Phi \colon F(n-1,2n) \to \Gamma$ with the property  
\[
\forall_{0 \leq i \leq n-2} \Phi(a_i) = D_i.
\]
\end{proof}

\section{The theorem}

In the following section we shall prove the main theorem of the article:
\begin{thm}
Let $n \geq 3$ be an odd integer. Let $\G \subset E(n)$ be a HW-group of dimension $n$. Then there exists an epimorphism
\[
\Phi \colon F(n-1, 2n) \to \G.
\]
\end{thm}

\begin{proof}
By Theorem \ref{thm:hwform} we can assume that
\[
\G = \langle (B_i, b_i) \st 0 \leq i \leq n-2 \rangle,
\]
where
\[
B_i = \diag(\underbrace{-1, \ldots, -1}_{i}, 1, -1, \ldots, -1)
\]
and $b_i \in \frac{1}{2}\Z^n$ for $i=0, \ldots, n-2$. 
Let
\[
id_\G = \varphi^{(1)} \oplus \ldots \oplus \varphi^{(n)}
\]
be the decomposition of the identity map on $\G$ as in Corollary \ref{cor:hwdecomp}. For every $0 \leq k \leq n-1$ the group $\varphi^{(k+1)}(\G)$ is a one dimensional crystallographic group, generated by the elements of the form \eqref{eq:onedimgens}. By Corollary \ref{cor:onedimepi} for every $1 \leq k \leq n$ we get an epimorphism $\Phi_k \colon F(n-1, 2n) \to \varphi^{(k)}(\G)$ such that
\[
\forall_{0 \leq i \leq n-2} \Phi_k(a_i) = \varphi^{(k)} (B_i, b_i).
\]
It is easy to see that
\[
\Phi = \Phi_1 \oplus \ldots \oplus \Phi_n
\]
is the desired homomorphism and the claim follows.
\end{proof}

\bibliographystyle{amsplain}
\bibliography{bibl}

\end{document}